\documentclass[12pt]{amsart}
\usepackage{amsmath,amsfonts,amsthm,amssymb,amscd}
\usepackage{color}
\usepackage{hyperref}
\usepackage{tikz-cd}
\usepackage{graphicx}
\usepackage{enumerate}
\usepackage{changepage}

\theoremstyle{definition}

\usepackage[iso]{umlaute}

\usepackage[matrix,arrow,curve]{xy}

\usepackage{boxedminipage}
\usepackage{epic}

\usepackage{longtable}
\usepackage{ifthen}
\usepackage{tabularx}

\usepackage{verbatim}

\def\A{\mathbb{A}}
\def\C{\mathbb{C}}

\def\R{\mathbb{R}}

\def\Z{\mathbb{Z}}

\def\P{\mathbb{P}}
\def\R{\mathbb{R}}

\def\g{\mathfrak{g}}
\def\h{\mathfrak{h}}

\def\h{\mathfrak{h}}

\def\FF{\mathcal{F}}
\def\FG{\mathcal{G}}

\def\FI{\mathcal{I}}

\def\FS{\mathcal{S}}

\def\so{\mathfrak{so}}

\def\sp{\mathfrak{sp}}

\def\sbe{\subseteq}

\def\la{\langle}
\def\ra{\rangle}

\newcommand{\Spec}{\operatorname{Spec}}

\newcommand{\rank}{\operatorname{rank}}

\newcommand{\epi}{\operatorname{epi}}
\newcommand{\Hom}{\operatorname{Hom}}

\newcommand{\Ad}{\operatorname{Ad}}

\newcommand{\Aut}{\operatorname{Aut}}
\newcommand{\ssim}{\operatorname{ss}}

\newcommand{\GL}{\mathrm{GL}}

\newcommand{\PGL}{\mathrm{PGL}}

\numberwithin{equation}{section}

\newcommand{\tab}

\newtheorem{thm}{Theorem}[section]

\newtheorem{cor}[thm]{Corollary}

\newtheorem{prop}[thm]{Proposition}
\newtheorem{lem}[thm]{Lemma}

\newtheorem{defn}[thm]{Definition}

\newtheorem{claim}{Claim}
\newtheorem*{claim*}{Claim}

\begin{document}

\title[Representation Variety of  Fuchsian groups]{{Representation Variety of Fuchsian Groups in SO(\lowercase{p},\lowercase{q})} }

\author{Krishna Kishore}

\email{venkatakrishna@iiserpune.ac.in}

\address{Venkata Krishna Kishore Gangavarapu,
Department of Mathematics,
Indian Institute of Science Education and Research PUNE,
Pune, Maharashtra
India, 411 008}

\begin{abstract}
We estimate the dimension of the variety of homomorphisms from  $\Gamma$ to $ SO(p,q)$ with Zariski dense image, where $\Gamma$ is a Fuchsian group, and $SO(p,q)$ is the indefinite special orthogonal group with signature $(p,q)$. 
\end{abstract}

\maketitle

\section{Introduction}\label{intro}
Let $\Gamma$ be a finitely generated group and let $G$ be a linear algebraic group defined over $\R$. The set of homomorphisms $\Hom(\Gamma,G(\R))$ coincides with the real points of the \textit{representation variety} $X_{\Gamma,G} :=$ Hom$(\Gamma,G)$. (We note here that by a \textit{variety} we mean an affine scheme of finite type over $\R$. In particular, we do not assume that it is irreducible or reduced). Let $X^{\epi}_{\Gamma,G}(\R)$ denote the Zariski closure of the set of representations $\Gamma \to G(\R)$ with Zariski-dense image.  The main goal of the paper is to show that the dimension of $X^{\epi}_{\Gamma,G}$ is roughly of the order $(1-\chi(\Gamma)) \dim G$, where $G(\R)$ is the (indefinite) special orthogonal group $SO(p,q)$ with signature $(p,q)$. Before stating our main result, we provide a very brief discussion of a recent work in this direction.

In a recent work,  Liebeck and Shalev \cite{LS} studied the representations of cocompact oriented Fuchsian groups of genus $g \geq 2$ and cocompact nonoriented Fuchsian groups of genus $g \geq 3$ in semisimple linear algebraic groups; this study is based on the results from the theory of finite quotients of Fuchsian groups. An explicit formula for the dimension of $\Hom(\Gamma,G)$ is given in the case where $G$ is a connected simple algebraic group over an algebraically closed field of \textit{arbitrary characteristic}. On the other hand, based on the deformation theory of Weil \cite{We}, Larsen and Lubotzky \cite{LL} studied the representation variety of cocompact oriented Fuchsian groups $\Gamma$ of genus $g \geq 0$ in almost-simple real algebraic groups. (From now onward, we mean by Fuchisan group a cocompact oriented Fuchsian group of genus $g\geq 0$). An estimate of the dimension of $X^{\epi}_{\Gamma,G}$, is given in the following cases: 
\begin{enumerate}
  \item
  $\Gamma$ is any Fuchsian group and $G$ is $SO(n)$, $SU(n)$, any split simple real algebraic group.
  \item
  $\Gamma$ is a $SO(3)$-dense Fuchsian group and $G$ is any compact simple real algebraic group. We note here, a Fuchsian group $\Gamma$ is said to be $H$-dense, where $H$ is an almost simple algebraic group, if there exists a homomorphism $\phi: \Gamma \rightarrow H$ with dense image and $\phi$ is injective for all finite cyclic subgroups of $\Gamma$ \cite[Definition 1.1]{LL}. 
\end{enumerate}
\noindent The dimension of $X^{\epi}_{\Gamma,G}$ in these cases turns out to be roughly of the order $(1-\chi(\Gamma)) \dim G$, where $\chi(\Gamma)$ is the Euler characteristic of $\Gamma$. The main goal of this paper is to show that a similar estimate of the dimension of $X^{\epi}_{\Gamma,SO(p,q)}$ holds, where $SO(p,q)$ is the indefinite special orthogonal group with signature $(p,q)$. Note that when $p = q$, or $p = q+1$, or $q = p+1$ the group $SO(p,q)$ is split; as mentioned earlier, the estimates in these cases are treated by Larsen and Lubotzky \cite{LL}. The main result of this paper deals with the remaining cases, that is when $p \neq q$, $p \neq q+1$, and $q \neq p+1$; in these cases $SO(p,q)$ is nonsplit and noncompact. Before stating the main result we recall some notation.

A cocompact oriented Fuchsian group $\Gamma$ admits a presentation of the following kind: consider non-negative integers $m$ and $g$ and integers $d_1,\ldots,d_m$ greater than or equal to $2$, such that the Euler characteristic
\begin{equation*} \label{cha}
\chi(\Gamma) : = 2-2g-\sum_{i=1}^{m}(1 - d_i^{-1})
\end{equation*}
\noindent is negative. For some choice of $m$, $g$, and $d_i$, $\Gamma$ has a presentation
\begin{equation}\label{pre}
\begin{split}
\Gamma  :=  \langle x_1,\ldots,x_m,y_1,\ldots,y_g,z_1,\ldots, & z_g  \; \vert \;  x_1^{d_1},\ldots,x_m^{d_m}, \\
 & x_1 \ldots x_m[y_1,z_1]\ldots[y_g,z_g] \rangle. 
 \end{split}
\end{equation}

\noindent Now we state the main theorem of the paper.

\begin{thm}\label{intro-thm-2}
For every Fuchsian group $\Gamma$ and for every special orthogonal group $SO(p,q)$, where $p \neq q$, $q \neq p+1$ and $p,q$ sufficiently large,
\begin{equation*}
 \dim   X^{\epi}_{\Gamma,SO(p,q)}  =  (1- \chi(\Gamma))   \; \dim  SO(p,q) + O(\rank SO(p,q) ),
\end{equation*}
where the implicit constant depends only on $\Gamma$.
\end{thm}

\noindent The proof of the theorem is based on the deformation theory of Weil \cite{We} and the work of Larsen and Lubotzky \cite{LL}. While the upper bound follows almost immediately from the work of Larsen and Lubotzky, establishing a lower bound is difficult. A brief sketch of the proof follows.

Let $\Gamma$ be a Fuchsian group with the presentation given by (\ref{pre}), and let $G$ be any linear algebraic group.  The Zariski tangent space at any point $\rho \in X_{\Gamma,G}$ is given by the space of $1$-cocycles $Z^1(\Gamma,\Ad \circ \rho)$, where $\Ad$ is the adjoint representation of $G$ in its Lie algebra $\g$. Let $\g^{*} = (\Ad \circ \rho)^{*}$ be the coadjoint representation of $\Gamma$. The dimension of $Z^1(\Gamma,\Ad \circ \rho)$ is given by the following formula \cite{We}:
\begin{equation}
\label{tas}
\begin{split}
\dim   Z^1(\Gamma,\Ad \circ \rho) &=  (1 - \chi(\Gamma))  \dim  G \;  +
   \dim (\g^{*})^\Gamma   \\ 
 &+ \sum_{i =1 }^{m} \left( \frac{ \; \dim   G }{d_i} -  \; \dim   \g^{\langle x_i \rangle} \right)
\end{split}
\end{equation}
\noindent where $(\g^{*})^\Gamma$ denotes the $\Gamma$-invariant vectors under the coadjoint representation of $\Gamma$. Based on this formula, if $\rho : \Gamma \to G(\R)$ is a representation with Zariski-dense image in $G(\R)$, an upper bound on $\dim Z^1(\Gamma,\Ad \circ \rho)$, is given by  \cite[Proposition 2.1]{LL}
\begin{equation}
\begin{split}
\dim Z^1(\Gamma,\Ad \circ \rho) \leq (1 - \chi(\Gamma)) \dim G &+ (2g + m + \rank G) \\
&+ \frac{3}{2} m \rank G.
\end{split}
\end{equation}
As each irreducible component $C$ of $X^{\epi}_{\Gamma,G}$ contains a point $\rho$ with Zariski dense image in $G(\R)$, the dimension of $C(\R)$ is bounded above by the dimension of $Z^1(\Gamma, \Ad \circ \rho)$. The upper bound of $\dim \; X^{\epi}_{\Gamma,G}$ follows.

On the other hand, to establish the lower bound, we construct a representation with Zariski-dense image in $SO(p) \times SO(q)$. As it turns out, $\rho_0$ is a nonsingular point in $X_{\Gamma,SO(p,q)}$, therefore it belongs to a unique irreducible component $C$ of $X_{\Gamma,SO(p,q)}$. The main emphasis of this paper is to show that that there exists an open neighborhood of $\rho_0$ (in the real topology), that does \textit{not} contain representations of the following types:

\begin{enumerate} 
\item 
Representations which stabilize some $d$-dimensional subspace where $d \neq p,q,\dim V$ or which stabilize a $p$-dimensional \textit{isotropic} subspace of $V$, where $V = \R^{p+q}$ denotes the natural representation of $SO(p,q)$.
\item 
Representations $\rho \in X_{\Gamma, SO(p,q)}$ such that the semisimple rank of the Zariski-closure of $\rho(\Gamma)$ is strictly less than the semisimple rank of $SO(p) \times SO(q)$. 
\end{enumerate}

\noindent Sieving representations of the former type is based on the theory of symmetric bilinear forms, and is easy. Sieving representations of the latter type is based on the theory of semisimple algebraic groups, and requires some work. In the end, we have the following result:

\begin{thm}\label{intro-neigh}
Let $\Gamma$ be a finitely generated group and let $\rho_0: \Gamma \to SO(p,q)$ be a representation of $\Gamma$ in $SO(p,q)$ with Zariski-dense image in $SO(p) \times SO(q)$. Suppose $p < q $, $q \neq p+1$, and $p,q$ sufficiently large. Then, there exists an open neighborhood of $\rho_0$ in $X_{\Gamma,SO(p,q)}$ (in the real topology), such that each point $\rho$ in the neighborhood satisfies precisely one of the following conditions:

\begin{enumerate}[(a)]
 \item
The Zariski closure of $\rho(\Gamma)$ is $SO(p,q)$.
\item 
The Zariski closure of $\rho(\Gamma)$ is conjugate to $SO(p) \times SO(q)$.
\end{enumerate}
\end{thm}

\noindent Now, using the deformation theory of Weil \cite{We}, we show that the representations of the type $(b)$ in the open neighborhood given by Theorem \ref{intro-neigh} constitute a proper closed subvariety, whence the representations of type $(a)$ constitutes a Zariski dense subset of the irreducible component $C$. It follows that $C$ is contained in  $X^{\epi}_{\Gamma,SO(p,q)}$ (recall $X^{\epi}_{\Gamma,SO(p,q)}$ is the Zariski-closure of the set of representations $\rho: \Gamma \to SO(p,q)$ with Zariski-dense image); hence $\dim X^{\epi}_{\Gamma,SO(p,q)}$ is bounded below by $\dim C$ which in turn can be estimated by estimating the dimension of the Zariski-tangent space at $\rho_0$ given by (\ref{tas}). Along the way to proving Theorem \ref{intro-neigh} we prove the following result which may be of independent interest.

\begin{prop}\label{intro-maximal-avoidance}\emph{(Maximal subgroup avoidance)} Let $\Gamma$ be a finitely generated group. Let $G$ be a semisimple subgroup of the special linear group $SL_n(\R)$. Let $K$ be a semisimple subgroup of $G$ and let $\rho_0 : \Gamma \to G$ be a representation of $\Gamma$ such that the Zariski-closure of $\rho_0(\Gamma)$ is $K$. Let $\{ H_1, \ldots, H_r \}$ represent a finite set of conjugacy classes of proper reductive subgroups of $G$.  Suppose $\rank  H_i < \rank  K$ for all $1\leq i \leq r$. Then there exists $\gamma \in \Gamma$ such that $\rho_0(\gamma) \notin gH_ig^{-1}$ for all $g \in G$ and $1 \leq i \leq r$. 
\end{prop}

Also, we shall present a proof of a generalisation of Jordan's theorem on finite groups in $GL_n(\C)$.

\begin{prop}\label{intro-larsen}
For all positive integers $n$ there exists a constant $J(n)$ such that every closed reductive subgroup $H$ of $\GL_n$,
has an open subgroup $H'$ of index $\le J(n)$ such that
$$H' = Z(H') [H^\circ,H^\circ]$$
where $H^\circ$ denotes the identity component of $H$.
\end{prop}

\noindent The paper is organized as follows. In \S \ref{term} we introduce some terminology and notation that shall be adopted throughout the paper. In \S \ref{sem-sim} we prove a lemma on the finiteness of number of maxmal proper closed subgroups of a reductive group; this lemma will be used in \S \ref{word-maps}.  In \S \ref{word-maps} we consider word maps on semisimple algebraic groups and prove Proposition \ref{intro-maximal-avoidance} that is useful in establishing a certain uniform criterion on the subgroups of reductive groups. In \S \ref{jordan} we provide a proof of Proposition \ref{intro-larsen}; the statement and the proof is due to Michael Larsen. In \S \ref{deformation-point} we begin the analysis of the representation variety $X_{\Gamma,SO(p,q)}$ and prove Theorem \ref{intro-neigh}. In \S \ref{main-thm} we prove the main result. In \S \ref{final-rmk} we comment on the possibility of extending these results to include other types of noncompact nonsplit real algebraic groups.

\section*{Acknowledgement}

It is a great pleasure to acknowledge my adviser, Michael Larsen, for many conversations during the course of this work. Also, I would like to thank him for the invaluable comments on the exposition of the work.

\section{Terminology and Notation}\label{term}

\tab All $Fuchsian$ groups in this paper are assumed to be cocompact and oriented. A \textit{variety} is an affine scheme of finite type over $\R$. Points are $\R$-points and nonsingular points should be understood scheme-theoretically. For a subset $Y$ of an affine algebraic subvariety $X$, we mean by $\overline{Y}$ the Zariski-closure of $Y$ in $X$.

An \textit{algebraic group} will mean a linear algebraic group over $\R$. The identity component of an algebraic group $G$ is denoted by $G^\circ$.  The semisimple rank of a subgroup $H$ in $G$ is denoted by $\rank_{\ssim} H$. The derived group of a subgroup $H$ is denoted by $[H,H]$.

We mean by $\R$-open subset, the open subset considered in the real topology. Unless otherwise stated, all topological notions should be understood in the Zariski topology. Of course, the notion of compactness should be understood in the real topology.

We introduce some terminology that shall be \textit{extensively} used throughout the paper. Let $V$ be a representation of $G$. Let $\rho : \Gamma \to G$ be a representation of $\Gamma$ in $G$; then $V$ is also a representation of $\Gamma$ via $\rho$. Informally speaking, a representation is said to have some property if the Zariski closure of its image has that property. 

We say that $\rho$ is \textit{dense} in $G$ if $\rho(\Gamma)$ is Zariski dense in $G$, and say that $\rho$ \textit{acts irreducibly} on $V$ if $\overline{\rho(\Gamma)}$ acts irreducibly on $V$.  Similarly, $\rho$ is said to be conjugate to a subgroup $H$ of $G$ if $\overline{\rho(\Gamma)}$ is conjugate to $H$. The \textit{rank of a representation $\rho$}, denoted $\rank \rho$, is defined to be the reductive rank of the Zariski closure of the image of $\rho$. The \textit{semisimple rank of a representation $\rho$}, denoted $\rank_{\ssim} \rho$, is defined to be the semisimple rank of the Zariski closure of the image of $\rho$.  We say $\rho$ is conjugate to a subgroup $H$ if the Zariski-closure of the image of $\rho(\Gamma)$ is conjugate to $H$.  Note that the symbol $\rho$ is loaded with several meanings, but we believe that the particular meaning will be clear from the context. 

Finally, for a vector space $V$ over $\R$, we mean by $V_\C$ the complexification of $V$ : $V_\C := V \otimes \C$.


\section{Maximal subgroups of semisimple real algebraic groups}\label{sem-sim}

The essential aspect in computing the dimension of the set of homomorphisms $\rho : \Gamma \to G(\R)$ with Zariski-dense image is to establish that there is an open subset (in the real topology) of such homomorphisms. The basic idea is to \textit{sieve} all those representations with image contained in a proper closed subgroup of $G(\R)$. So, it is desirable to have some control over the maximal proper closed subgroups of $G(\R)$; semisimple algebraic groups offer such control, as the following proposition illustrates:

\begin{prop}
\label{maximals}
Let $G$ be an almost simple real algebraic group.  There exists a finite set $\{H_1,\ldots,H_k\}$ of proper closed subgroups of $G$ such that every proper closed subgroup
is contained in some group of the form $g H_i g^{-1}$, where $g\in G(\R)$.
\end{prop}
\begin{proof}
We refer the reader to \cite[Proposition 3.2]{LL}.
\end{proof}

\begin{lem}\label{finite-semisimple}
Let $G$ be a semisimple real algebraic group and let $V$ be a faithful irreducible representation of $G$. Then, there exists a finite set $\{ H_1, \ldots, H_n \}$ of proper semisimple subgroups of $G$, such that $H_i$ acts irreducibly on $V$ for each $i$, and every proper semisimple subgroup that acts irreducibly on $V$ is contained in some subgroup of the form $gH_i g^{-1}$, where $g \in G$. 
\end{lem}

\begin{proof}
Let $\{ K_1, \ldots, K_n \}$ represent the set of conjugacy classes of maximal proper closed subgroups of $G$. Let $H$ be a proper semisimple subgroup of $G$ and suppose $H$ acts irreducibly on $V$. Then $H$ is contained in some subgroup of the form $g K_i g^{-1}$, where $1 \leq i \leq n$ and $g \in G$. It is well-known that a linear algebraic group with a faithful irreducible representation is reductive, so $K_i$ cannot be parabolic, whence it must be reductive \cite[Theorem 30.4 (a)]{Hu}.  As $H$ is semisimple, it is contained in the semisimple subgroup $[K_i^\circ, K_i^\circ]$ of $K_i$ which is of strictly smaller dimension than $\dim G$. Now, the lemma follows by induction on dimension.
\end{proof}

\section{Word maps on semisimple algebraic groups}\label{word-maps} Let $F$ be a free group on $m$ letters $X_1, \ldots,X_m$, and let $w = w(X_1,\ldots,X_m)$ ($m \geq 2$) be a reduced word in the $X_i$'s with non-zero  exponents. Then, given an algebraic group $G$, the word $w$ defines a map $f_w : G^m \to G$, by the rule
$$
f_w( g_1, \ldots, g_m ) = w(g_1,\ldots,g_m).
$$

\begin{thm}\label{borel}
(Borel): The morphism $f_w : G^m \to G$ is dominant. 
\end{thm}
\begin{proof}
We refer the reader to \cite{Bo1}.
\end{proof}

\noindent Let $H$ be a closed subgroup of $SL_n(\R)$. Given $h \in H$, let $X^n + a_1(h) X^{n-1}+ $ $ \ldots + a_{n-1}(h) X + (-1)^n = 0 $ be its characteristic polynomial, where $a_1(h), \ldots $, $a_{n-1}(h) \in \C$. Consider the \textit{characteristic morphism} of affine varieties $\chi : H \to \A^{n-1}_{\C}$ defined by
$$
\chi(h) = (a_1(h), \ldots,a_{n-1}(h)),
$$ where $\A_{\C}^{n-1}$ is the affine space $\C^{n-1}$. Clearly, the morphism $\chi$ is constant on conjugacy classes of elements of $H$.  

\begin{prop}\label{Image-rank}
Let $H$ be a semisimple subgroup of $SL_n(\R)$. Then, $\dim \overline{\chi(H)}$ = $\rank H$.
\end{prop}

\begin{proof}
The dimension of a variety is invariant under the base extension from $\R$ to $\C$, so without loss of generality we may consider $H(\C)$ in $SL_n(\C)$. It is well-known that any maximal torus $T$ of $H$ extends to a maximal torus of $SL_n(\C)$, i.e. $T$ is a subtorus of some maximal torus of $SL_n(\C)$. Furthermore, since any two maximal tori of $SL_n(\C)$ are conjugate in $SL_n(\C)$, we may further assume that $T$ is a subtorus of the maximal torus in $SL_n(\C)$ consisting of diagonal matrices.

The set of regular semisimple elements of $H$ contains a nonempty open subset $U$ of $H$ \cite[Theorem 12.3]{Bo}. Since any regular semisimple element is conjugate to an element of $T$, we have the following inclusions:
$$
\chi(T) \subseteq \chi(H) = \chi(\overline{U}) \subseteq \overline{\chi(U)} \subseteq \overline{\chi(T)}.
$$
Hence, without loss of generality we may consider $\chi : T \to \overline{\chi(T)}$. Since $\chi$ is dominant, there exists a nonempty open subset $W$ of $\overline{\chi(T)}$ contained in $\chi(T)$ such that for each $y \in W$,
$$
\dim  \chi^{-1}(y) =  \dim T - \dim \overline{\chi(T)}.
$$
The fiber over $y$ is a conjugacy class in $T$ represented by a diagonal matrix. Since any conjugate of a diagonal matrix is obtained by permuting the diagonal entries, it follows that the fiber over $y \in W$ consists of at most $n!$ elements, whence dim $\chi^{-1}(y) = 0 $ for each $y \in W$. The proposition follows from this, since $\dim W = \dim \overline{\chi(T)}$.
\end{proof}

\begin{prop}\label{Maximal-avoidance}\emph{(Maximal subgroup avoidance)} Let $\Gamma$ be a finitely generated group. Let $G$ be a semisimple subgroup of the special linear group $SL_n(\R)$. Let $K$ be a semisimple subgroup of $G$ and let $\rho_0 : \Gamma \to G$ be a representation of $\Gamma$ such that $\overline{\rho_0(\Gamma)} = K$. Let $\{ H_1, \ldots, H_r \}$ represent a finite set of conjugacy classes of proper reductive subgroups of $G$.  Suppose $\rank  H_i < \rank  K$ for all $1\leq i \leq r$. Then there exists $\gamma \in \Gamma$ such that $\rho_0(\gamma) \notin gH_ig^{-1}$ for all $g \in G$ and $1 \leq i \leq r$. 
\end{prop}

\begin{proof}
By Lemma \ref{Image-rank}, the dimension of image of the semisimple subgroup $S_i := [H_i^\circ,H_i^\circ]$ with respect to $\chi$ is equal to the rank of its maximal torus. So, by hypothesis on the ranks of $H_i$ and $K$, it follows that the dimension of $\mathcal{S} := \bigcup_{i=1}^{r} \overline{\chi(S_i)}$ is strictly less than the dimension of $\mathcal{T} := \overline{\chi(K)}$.

For each $1 \leq i \leq r$, let  $m_i$ be the number of connected components of $H_i$, and let $m$ be their least common multiple. Then, for each $i$ ($1 \leq i \leq r$), and each $x_i , y_i \in H_i$, the commutator $[x_i^m,y_i^m] \in [H_i^\circ,H_i^\circ]$. Consider the following composition of morphisms:
\[
 \Gamma \times \Gamma \xrightarrow{\rho_0 \times \rho_0} G \times G  \xrightarrow{f_m}  G \xrightarrow{\chi}  \A^{n-1}
\]
where $f_m$ is defined by $f_m(x,y) = x^m y^m x^{-m} y^{-m}$. We claim that the composite morphism $\chi \circ f_m \circ (\rho_0 \times \rho_0)$ has dense image in $\mathcal{T}$. Indeed, as $\rho_0(\Gamma)$ is dense in $K$, $(\rho_0 \times \rho_0)(\Gamma \times \Gamma)$ is dense in $K \times K$. By Theorem \ref{borel}, the morphism $f_{m} \mid_{K \times K}$ is dominant, so there exists a dense open subset $U$ in $K$ contained in $f_m(K \times K)$. Since $\chi(K)$ is also also dense in $\mathcal{T}$, it follows that the image of the composite morphism $ \chi \circ f_m \circ (\rho_0 \times \rho_0)$ is dense in $\mathcal{T}$. 

On the other hand, since $\dim \mathcal{S} < \dim \mathcal{T}$, there exists $z \in \mathcal{T} \setminus \mathcal{S}$, such that $ z = \chi(\gamma)= \chi([x^m,y^m])$ and $[x^m,y^m] \in U$,  where $x  = \rho_0(\alpha), y = \rho_0(\beta)$ for some $\alpha,\beta \in \Gamma$. We claim that the element $\gamma \in \Gamma$ does not belong to $g H_i g^{-1}$ for each $i$ and for each $g \in G$. Indeed, if $\rho_0(\gamma) \in g H_i g^{-1}$ for some $i$ and for some $g \in G$, then $f_m  \circ (\rho_0 \times \rho_0) (\alpha, \beta) = \rho_0(\gamma) \in gH_i g^{-1}$, and therefore its image under $\chi$ belongs to $\mathcal{S}$, contrary to what is established: $ z =\chi \circ \rho_0(\gamma)  \in  \mathcal{T} \setminus \mathcal{S}$. 
\end{proof}

\section{A generalisation of a theorem of Jordan}\label{jordan}

The following proposition is a generalisation to reductive subgroups of Jordan's theorem on finite subgroups of $GL_n(\C)$. All algebraic groups will be defined over $\C$. We refer to \cite{La} for a proof for compact groups. 

\begin{lem}
\label{strong-Jordan}
Let $a_1,\ldots,a_k$ be positive integers, $G = \GL_{a_1}\times\cdots\times \GL_{a_k}$.
Let $C$ denote  a closed subgroup of the center of $G$, $\bar G$ the quotient group $G/C$, and
$K$ a finite subgroup of $\bar G(\C)$.  Then there exists a subgroup $K'$ of $K$ with $|K/K'|$
bounded above by a bound depending only on the $a_i$, such that the inverse image of $K'$ in 
$G$ is commutative.
\end{lem}

\begin{proof}
It suffices to prove the lemma after replacing $C$ by $Z(G)$, $\bar G$ by $G/Z(G) = \PGL_{a_1}\times \cdots\times \PGL_{a_k}$,
and $K$ by its image in $G/Z(G)$.  Without loss of generality, therefore, we assume $C=Z(G)$.  Embedding each factor
$\PGL_{a_i}$ by its adjoint representation into $\GL_{a_i^2-1}$, we obtain a faithful representation $\bar G\to \GL_N$
where $N = \sum_i a_i^2-1$.  By Jordan's theorem, $K$ has an abelian subgroup whose index is bounded in terms of $N$.
Without loss of generality, therefore we assume $K$ is commutative.  It follows that for every $d$, the group $K_d := \{h^d\mid h\in K\}$ is of index
$\le d^N$ in $K$.  
If $u,v\in \PGL_{a_i}(\C)$ commute, then for any lifts $\tilde u$ and $\tilde v$ of $u$ and $v$ to $\GL_{a_i}(\C)$, 
$$\tilde u \tilde v \tilde u^{-1} = \zeta \tilde v,$$
where $\zeta^{a_i} = 1$.  Thus $\tilde u^{a_i}$ and $\tilde v^{a_i}$ commute.  Applying this to the projections of $x$ and $y$ to varying
factors $\PGL_{a_i}(\C)$, we conclude that $\tilde x^d$ and $\tilde y^d$ commute for all lifts $\tilde x$ and $\tilde y$ in $G$, where
$d$ denotes the least common multiple of the $a_i$.  Thus we may take $K' = K_d$.
\end{proof}

\begin{prop}\label{larsen}\footnote{The proof is also due to Michael Larsen.} 
For all positive integers $n$ there exists a constant $J(n)$ such that every closed reductive subgroup $H$ of $\GL_n$,
has an open subgroup $H'$ of index $\le J(n)$ such that
$$H' = Z(H') [H^\circ,H^\circ].$$ 
\end{prop}

\begin{proof}
 Let $D = [H^\circ,H^\circ]$ be the semisimple part of $H$ and $Z = Z(H^\circ)$.  
The conjugation action of $H$ on its characteristic subgroup $Z^\circ$ factors through 
$H/H^\circ$ to give a homomorphism $H/H^\circ\to \Aut(Z^\circ) = \GL_r(\Z)$, where $r$ is the rank of the torus  $Z^\circ$.
It is well known that finite subgroups of $\GL_r(\Z)$ are bounded in terms of $r$ (thus in terms of $n$), so we may assume
$Z^\circ$ lies in the center of $H$.  

The conjugation action of $H$ on its characteristic  subgroup 
$D$ defines a homomorphism $H\to \mathrm{Out}(D)$.  As $\mathrm{Out}(D)$ can be bounded in terms of the rank of $D$ (and therefore in terms
of $n$), we are justified in assuming that $H$ acts by inner automorphisms on $D$.  
Thus, 
$$H\subset D Z_{\GL_n}(D Z^\circ).$$

The product map 
$$D\times Z_{\GL_n}(D Z^\circ)\to D Z_{\GL_n}(D Z^\circ)$$
has kernel $Z(D)$ so there is a well-defined map 
$$\phi\colon H\to Z_{\GL_n}(D Z^\circ)/Z(D)$$
with kernel $D$, and therefore a well-defined map 
$$H\to Z_{\GL_n}(D Z^\circ)/Z(D) Z^\circ$$
with kernel $H^\circ = DZ^\circ$.

As $D$ is semisimple, the natural $n$-dimensional representation of $D$ is semisimple, hence of the form 
$$V_1^{\oplus a_1}\oplus\cdots\oplus V_k^{\oplus a_k}$$
for pairwise distinct irreducible factors $V_i$. The centralizer of $D$ in $\GL_n$  is therefore
$$G = \GL_{a_1}\times\cdots\times \GL_{a_k},$$
where $a_1+\cdots+a_k \le n$, so the set of possibilities for the $a_i$ is bounded in terms of $n$.
Setting $C = Z(D) Z^\circ$ and $K = H/H^\circ$ and applying Lemma~\ref{strong-Jordan}, it follows that $H/H^\circ$ has 
a subgroup $K'$ of bounded index, which then determines an open subgroup $H'$ of bounded index in $H$.
Thus, $H' = D Z_{H'}(D)$, and by Lemma~\ref{strong-Jordan}, $Z_{H'}(D)$ is commutative.  Finally,
$Z_{H'}(D)$ commutes with $D$, so it lies in $Z(H')$.
\end{proof}

\section{Neighborhood of a deformation point}\label{deformation-point}

\tab In this section we reserve the notation $\Gamma$ to denote any finitely generated group and $G$ to denote the (indefinite) special orthogonal group $SO(p,q)$. Also, throughout the section, we assume $p < q$, $q \neq p+1$, and $p,q$ sufficiently large. Let $(V,B)$ denote the natural representation of $SO(p,q)$, where $V = \R^{p+q}$ and $B: V \times V \to \R$ is the bilinear form. Let $Q(v) := B(v,v)$  denote the quadratic form  associated to $B$. (Recall, the stabilizer of $Q$ under the natural action of $GL_{p+q}(\R)$ on $V$ is $O(p,q)$, the indefinite orthogonal group with signature $(p,q)$.)


It is not true, in general, that the set of $\R$-points of a variety $X$ defined over $\R$ is nonempty. But, in the case where $X$ contains a nonsingular real point, there exists an $\R$-open neighborhood  diffeomorphic to $\R^{d}$, where $d$ is the dimension of the Zariski-tangent space at the point. The proof is essentially based on the Implicit Function Theorem employed in the following well-known result. Given a lack of proper reference we provide a proof.

\begin{thm}\label{real-points}
Let $X$ be an affine scheme of finite type over $\R$ and let $x$ be a nonsingular real point in $X$. Then there is an $\R$-open neighborhood of $x$ diffeomorphic to $\R^{d}$, where $d$ is the dimension of the Zariski-tangent space of $X$ at $x$.
\end{thm}

\begin{proof}
By hypothesis $X = \Spec A$, where $A$ is finitely generated  $\R$-algebra. Then $A \cong \R[X_1, \ldots, X_n] /I$ for some ideal $I$, and the $\R$-points are given by 
$$
X(\R) = \{ (a_1, \ldots, a_n) \in \R^n: f(a_1, \ldots, a_n) = 0, \; \mathrm{for \; each \; f \in I} \}.
$$ 
Since $x$ is a nonsingular real point, it belongs to a unique irreducible component of $X$; then there exists some $f_1, \ldots, f_{n-d}$ such that the following conditions hold \cite[Chapter III, \S 4, Corollary 1]{Mu}:
\begin{enumerate} 
\item 
$f_1, \ldots, f_{n-d}$ generates $I$ in a neighborhood of $x$.
\item 
the Jacobian matrix $\left( \frac{\partial f_i}{\partial x_j} \right)_{1 \leq i,j \leq n-d}$ is invertible at $x$.
\end{enumerate}
By the Implicit Function Theorem, there exists a neighborhood $U$ of $x$ in $X(\R)$, a neighborhood $V$ of the origin in $\R^{d}$, and a smooth function $g : V \to \R^{n-d}$ such that $\textrm{id} \times g : V \to \R^{n}$ maps onto $U$. Furthermore, $\pi \circ g = 1$, where $\pi: \R^{n} \to \R^{d}$ is the usual projection map. This gives the desired diffeomorphism. 
\end{proof}

\begin{prop}\label{R-open-nbhd}
Let $\rho_0 \in X_{\Gamma,G}$ be a representation with Zariski-dense image in $SO(p) \times SO(q)$. Then $\rho_0$ is a nonsingular point in $X_{\Gamma,G}$.  Moreover, there exists an open neighborhood of $\rho_0$ diffeomorphic to $\R^{d}$, where $d$ is the dimension of the Zariski-tangent space at $\rho_0$.
\end{prop}

\begin{proof}
Let $H := SO(p)\times SO(q)$. Since $\g$ is a self-dual $G$-representation, via the Killing form, the $\Gamma$-invariant vectors under the coadjoint representation of $\g$ are given by $(\g^{*})^\Gamma = \g^\Gamma = \g^H$. By a result of Weil \cite{We}, if the coadjoint representation of $\Gamma$ in $\g$ has no $\Gamma$-invariant vectors then $\rho_0$ is a nonsingular point in $X_{\Gamma,G}$. Therefore $\rho_0$ belongs to a unique irreducible component $C$ of $X_{\Gamma,G}$, and furthermore by Theorem \ref{real-points},  $C$ contains an $\R$-open neighborhood of $\rho_0$ diffeomorphic to $\R^{d}$, where $d$ is the dimension of the Zariski-tangent space at $\rho_0$. It remains to show that $\g^{H} = \{ 0\}$, equivalently that the centralizer $Z_H(G)$ of $H$ in $G$ is finite. 

Let $M(m,n)$ denote the $\R$-vector space of matrices of size $m \times n$. Let $GL_n(\R)$ denote the group of invertible matrices of $M(n,n)$. Let $I_n$ denote the identity matrix of size $n \times n$. Let 
$S=
\begin{bmatrix}
A&B\\
C&D
\end{bmatrix}
\in G$, where $A \in M(p,p)$, $B \in M(p,q)$, and $C \in M(q,p)$, $D \in M(q,q)$.
Let 
$T \in 
\begin{bmatrix}
M&0\\
0&N
\end{bmatrix}
\in SO(p) \times SO(q)$, where $M \in M(p,p)$, $N \in M(q,q)$. 

The condition $S \in Z_G(H)$ is equivalent to the relations $AM = MA$, $DN = ND$, $BN = MB $ and $CM = NC$ for each $M \in SO(p)$, $N \in SO(q)$. Since $SO(n)$ acts irreducibly on its natural representation $\R^n$ ($n\geq 3$), the relations $BN = MB$ and $CM = NC$ imply that $B = 0$ and $C = 0$ (take $N = I_q$, $M = I_p$). It follows that $A,B$ are invertible, whence from the first two relations we have $A \in Z_{GL(p)}(SO(p))$,  $B \in Z_{GL(q)}(SO(q))$. Since the centralizer of $SO(n)$ in $GL(n)$ consists of scalar matrices whenever $n \geq 3$, it follows that $A = \lambda I_{p}$, and $B = \mu I_{q}$, where $\lambda, \mu  \in \R^{*}$. Hence $S$ must be of the form $\lambda I_p \oplus \mu I_q$. But $S$ also preserves the the canonical bilinear form $I_p \oplus (-I_q)$ associated to $V$, so $\lambda = \pm 1$ and $\mu = \pm 1$. The result follows.
\end{proof}


Now, we properly begin the analysis of the neighborhood of the \textit{deformation point} $\rho_0 : \Gamma \to SO(p,q)$ with Zariski-dense image in $SO(p) \times SO(q)$. 
First, let us recall some definitions from the theory of quadratic spaces. 

\begin{defn} 
\begin{enumerate}[(a)]
\item 
A vector $ v \in V$ is said to be \emph{isotropic} if $Q(v)= 0$ and $\mathit{anisotropic}$ otherwise.
\item A subspace $W \sbe V$ is said to be \emph{isotropic} if it contains a nonzero isotropic vector.

\item
A subspace $W \sbe V$ is said to be \emph{totally-isotropic} if $B(w,w') = 0$ for all $w,w' \in W$.

\item 
A vector $v \in V$ is \emph{positive} if $Q(v) > 0$, and a subspace $W \sbe V$ is positive if every nonzero vector of $W$ is positive. Similarly, the notions negative vector and negative subspace are defined.
\end{enumerate}
\end{defn}

In order to state the results succinctly, let us note here two properties of representations or points in $X_{\Gamma,G}$. \\

$(P_m)$: Let $\Delta$ be the intersection of all subgroups of index at most $m$ in $\Gamma$. Then $\rho \mid_{\Delta}$ acts absolutely irreducibly on $V$. \\

$(Q_m)$: Let $\Delta$ be the intersection of all subgroups of index at most $m$ in $\Gamma$. Then $\rho \mid_{\Delta}$ stabilizes some $p$-dimensional positive subspace $W$ of $V$ and its $q$-dimensional orthogonal complement $W^\perp$ which is a negative subspace of $V$, and does not stabilize any proper nonzero subspace other than $W$ and $W^\perp$. Also, $\rho \mid_{\Delta}$ acts absolutely irreducibly on $W$ and $W^\perp$.  \\

\noindent Note that the intersection is over finitely many subgroups, since there are only finitely many subgroups of a given finite index in any finitely generated group. 

It is well known that the set $Gr(k,V)$ of $k$-dimensional subspaces of $V$ has a natural embedding in $\P(\bigwedge^{k}V)$ as a closed subvariety (called the \textit{Grassmannian variety}). Informally, in some fixed basis of $V$, a $k$-dimensional subspace is given in Pl$\ddot{\textrm{u}}$cker coordinates, and the variety $Gr(k,V)$ is the vanishing set of the polynomial relations among the coordinates (called Pl$\ddot{\textrm{u}}$cker relations). The \textit{flag variety} $\FF(V;n_1,\ldots,n_r)$, where $1 \leq n_1 < n_2 < \ldots < n_r \leq \dim V$, is defined as the set of all chains of subspaces $V_i$ of $V$:
\[
\FF(V;n_1,\ldots,n_r) = \{ 0 \sbe V_1 \sbe V_2 \sbe \ldots V_r \sbe V   \colon   \; \dim   V_i = n_i \},
\]
It is well known that $\FF(V;n_1,\ldots,n_r)$ also has a natural embedding as a closed subvariety in $Gr(n_1,V) \times Gr(n_2,V) \times \ldots \times Gr(n_r,V)$.

\begin{lem}\label{compactness}
Let $k$ be a fixed positive integer strictly less than $\dim V$. The following properties of $\rho \in X_{\Gamma,G}$  are closed in the real topology:
\begin{enumerate} 
\item 
$\rho$ stabilizes a $k$-dimensional subspace of $V_{\C}$.
\item
$\rho$ stabilizes a $p$-dimensional isotropic subspace of $V$.
\end{enumerate} 
\end{lem}
\begin{proof}
First consider property  $(1)$. Let $\rho_1,\rho_2,\ldots,$ be a sequence of representations in $X_{\Gamma,G}$ converging to $\rho$. For each $n \geq 1$, suppose $\rho_n$ stabilizes some $k$-dimensional subspace $W_n$ of $V_\C$. Since $Gr(k,V_\C)$ is a compact submanifold of some projective space, we have $W_n \to W$, for some $W \in Gr(k,V_{\C})$. Clearly, the action of $G$ on $Gr(k,V_{\C})$ is continuous in the real topology, so if $g_n \to g$, $W_n \to W$, and $g_n$ stabilizes $W_n$ for each $n$, then $g$ stabilizes $W$. In particular, this holds if $\rho_n(\gamma) \to \rho(\gamma)$, where $\gamma$ is a generator of $\Gamma$ (refer to equation \ref{pre} of \S \ref{intro}). As $\rho$ determines and is determined by $\rho(\gamma)$, the property  that $\rho \in X_{\Gamma,G}$ stabilizes a $k$-dimensional subspace of $V$ is closed in the real topology.

On the other hand, let $\FG(V;1,p)$ be the subset of $\FF(V;1,p)$ consisting of flags $ 0 \sbe V_1 \sbe V_p  \sbe V $ such that $W$ is spanned by an isotropic vector. Consider the canonical projection, 
$$
\pi_1 : Gr(1,V) \times Gr(p,V) \to Gr(1,V).
$$
Let $Z(Q)$ denote the variety defined by the quadratic form $Q$ in $Gr(1,V)$.  Then, $\FG(V;1,p) = \pi_1^{-1}(Z(Q)) \cap \FF(V;1,p)$, so that $\FG(V;1,p)$ is a compact submanifold in some projective space. Now, the proof of property $(2)$ is similar to the proof of property $(1)$.
\end{proof}

\begin{prop}\label{Deformation-point}
Let $\rho_0  :\Gamma \to G$ be a representation with Zariski dense image in $SO(p) \times SO(q)$. For each positive integer $m$,  there is an $\R$-open neighborhood $U$  of $\rho_0$ such that each $\rho \in U$ satisfies either the property $(P_m)$ or the property $(Q_m)$.
\end{prop}
\begin{proof}

First, let us show that the statement holds with $m=1$, i.e., there is a neighborhood $U$ of $\rho_0$ such that for each $\rho \in U$ satisfies either the property $(P_1)$ or the property $(Q_1)$. Given a representation $\rho \in X_{\Gamma,G}$, let $k_\rho$ denote the minimum of the dimensions of all subspaces stabilized by $\rho$. Let $\FS_k$ be the set of representations which stabilize some $k$-dimensional subspace of $V$. If $\rho$ stabilizes some $k$-dimensional subspace of $V$ then it also stabilizes its $(p+q-k)$-dimensional orthogonal complement, so the set of representations $\rho$ with $k_\rho \neq p$ and $k_\rho \neq \dim V$ is given by
$$
\FS :=  \FS_1 \cup \ldots \FS_{p-1} \cup \FS_{p+1} \; \cup \; \ldots \cup \FS_{q-1} $$ 
which, by Lemma \ref{compactness} $(1)$, is a closed subset in the real topology  (note that by assumption, $p < q$ and $q \neq p+1$). Clearly, $\rho_0 \notin \FS$. 

On the other hand, since $\rho_0$ does not stabilize any isotropic subspace, it follows from Lemma \ref{compactness} $(2)$ that $\rho_0$ does not belong to the $\R$-closed subset $\FI_p$ consisting of representations each of which stabilizes some $p$ dimensional isotropic subspace of $V$. Hence $U_\Gamma  := X_{\Gamma,G} \setminus \{ \FS \; \cup \; \FI_p \}$ is an $\R$-open neighborhood of $\rho_0$. By Lemma \ref{compactness} $(1)$, again, we see that the $\rho_0$ satisfies the latter statement on absolute irreducibility in the definition of property $(Q_m)$, and therefore shrinking $U_\Gamma$ if necessary we may assume that it consists of representations which satisfy either the property  $(P_1)$ or the property  $(Q_1)$.

Now, let us prove the statement for general $m$. Let $\Lambda$ be the intersection of all subgroups of index at most $m$ in $\Gamma$. Then $\rho_0(\Lambda)$ is Zariski-dense in $SO(p) \times SO(q)$. Proceeding as above, we obtain an $\R$-open neighborhood $U_\Lambda$ of $\rho_0 \mid_{\Lambda}$ in $X_{\Lambda,G} = \{ \rho \mid_\Lambda \colon \rho \in X_{\Gamma,G} \}$, such that each representation in $U_\Lambda$ either acts irreducibly on $V$ or stabilizes a $p$-dimensional positive subspace $W$ of $V$ and its $q$-dimensional orthogonal complement $W^\perp$ which is a negative subspace of $V$, and does not stabilize any proper nonzero subspace other than $W$ and $W^\perp$. Now, consider the canonical restriction morphism $\pi: X_{\Gamma,G} \to X_{\Delta,G}$, which is clearly $\R$-continuous.

The open neighborhood $U_{\Gamma} \cap \pi^{-1}(U_\Delta)$ of $\rho_0$ is the required neighborhood, where $U_{\Gamma}$ is the open neighborhood about $\rho_0$ chosen above.
\end{proof}

Now, we establish a result which uses results from the representation theory of semisimple Lie algebras over $\R$.

\begin{prop}\label{Rank-irreducible}
Let $H$ be a maximal proper closed subgroup of $G$. Suppose $H$ acts absolutely irreducibly on $V$. Then $\rank H < \rank SO(p) \times SO(q)$. (In particular, $\rank_{\ssim} H < \rank_{\ssim} SO(p) \times SO(q) = \rank SO(p) \times SO(q)$). 
\end{prop}

\begin{proof}

Let $\h$ be the Lie algebra of $H$. Without loss of generality, we may pass on to the level of Lie algebras and suppose that $\h$ acts absolutely irreducibly on $V$. Maximal subalgebras of $\g$ that act absolutely irreducibly on $V$ are either simple or non-simple. First, consider the maximal simple subalgebras $\h$ that are irreducible on $V_\C$. If a Lie algebra $\h$ over $\R$ is simple, then either $\h^{\C}$ is simple over $\C$ or $\h$ is simple Lie algebra over $\C$ considered as an algebra over $\R$. Moreover, in the latter case $\h^{\C} \cong \h \times \h$. First consider the case where $\h$ is simple over $\R$ and $\h^{\C}$ is simple over $\C$. By the proof of \cite[Theorem 1]{GLM}, every non-trivial irreducible representation of a simple Lie algebra of rank $r$, other than the natural representation and its dual, has dimension at least ${r \choose 2} -1$. Since 
$$
\dim V \geq 2 \cdot \lfloor \frac{p+q}{2} \rfloor  \geq 2 \cdot \rank \h
$$
and $\dim V$ is strictly less than ${r \choose 2} -1$ when $r \geq 6$, it follows that $V$ \emph{cannot} be the irreducible representation of $\h$. Now, consider the case where $\h$ is a simple Lie algebra over $\C$ considered as an algebra over $\R$; these subalgebras are given by \cite[Theorem $3$]{Ta}: $\so(m,\C) \subseteq \so(\frac{m(m-1)}{2}, \frac{m(m+1)}{2})$, $\sp(m,\C) \subseteq \so(m(2m-1),m(2m+1))$. These subalgebras have strictly lower rank than the $\rank$ of $\so(p) \oplus \so(q)$.

On the other hand, non-simple subalgebras that act irreducibly on $V_{\C}$ are classified by Taufik \cite[Theorems 1,3,4]{Ta}:
\begin{enumerate}[(a)]
\item 
$\so(p_1,q_1) \otimes I + I \otimes \so(p_2,q_2) $, where $ p = p_1 p_2 + q_1 q_2, q = p_1 q_2 + p_2 q_1$.
\item
$ \so^{*}(2n_1) \otimes I + I \otimes \so^{*}(2n_2)$, where $p = q = 2n_1n_2$.
\item
$\sp(m_1, \R) \otimes I  + I \otimes \sp(m_2,\R)$, where $ p= q = 2m_1m_2$.
\item
$\sp(p_1,q_1) \otimes I  + I \otimes \sp(p_2,q_2)$, where $ p= 4(p_1 p_2 + q_1 q_2) , q = 4(p_1 q_2 + p_2 q_1 )$.
\end{enumerate}

\noindent Recall the assumption (at the beginning of the section) that $p \neq q$, so cases $(b)$ and $(c)$ cannot arise. In the remaining cases $(a)$ and $(d)$, the ranks of the subalgebras are at most $\lfloor \frac{p_1 + q_1}{2} \rfloor + \lfloor \frac{p_2+ q_2}{2} \rfloor $ and $(p_1 + q_1) + (p_2 + q_2)$ respectively, which are clearly strictly less than $\lfloor (p_1+q_1)(p_2+q_2)/2 \rfloor$, when $p_1+q_1, p_2 +q_2$ are sufficiently large.
\end{proof}

\begin{lem} \label{absolute-reducible}
Let $K$ be a proper connected subgroup of $SO(p) \times SO(q)$. Suppose $\rank K = \rank SO(p) \times SO(q)$. Suppose $K$ stabilizes some $p$-dimensional positive subspace $W$ of $V$ and its $q$-dimensional orthogonal complement $W^{\perp}$ which is a negative subspace, and does not stabilize any proper nonzero subspace of $V$ other than $W$ and $W^{\perp}$. Then $K$ acts reducibly either on $W_\C$ or $(W^{\perp})_\C$. (Recall, $V$ denotes the natural representation of $SO(p,q)$).
\end{lem}

\begin{proof}
For ease of reference, let us refer to a subgroup of $SO(p) \times SO(q)$ of rank equal to the $\rank$ of $SO(p) \times SO(q)$, as a subgroup of \textit{full rank} in $SO(p) \times SO(q)$.

From \cite[Table 5 and Table 6]{OV}, it follows that up to conjugacy the only maximal proper connected subgroups in $SO(p)$ of full rank are $SO(k) \times SO(p-k)$, and $U(p/2)$ when $p$ is even. Clearly, subgroups $SO(k) \times SO(p-k)$ act reducibly on $W$, while the latter subgroup $U(p/2)$ acts irreducibly on $W$, and therefore the non-maximal subgroups of full rank of $SO(p)$ that act irreducibly on $W$ must be contained in $U(p/2)$ (up to conjugacy). Furthermore, the only maximal connected subgroups of full rank in $U(p/2)$ are of the form $U(k) \times U(p/2-k)$; clearly, they act reducibly on $W$, so it follows then that there are no non-maximal proper connected subgroups of $SO(p)$ of full rank that also act irreducibly on $W$.

Now, let $H := SO(p) \times SO(q)$. Let $K$ be a proper subgroup of $H$ of full rank. Consider the canonical projections $\pi_1 : H \to SO(p)$ and $\pi_2 : H \to SO(q)$.
If the image of $K$ under $\pi_1$, resp. $\pi_2$,  is a proper subgroup of $SO(p)$, resp., $SO(q)$, then it follows (by the above paragraph) that $K$ cannot act absolutely irreducibly on $W$, resp. $W^\perp$. On the other hand, if $K$ maps surjectively onto both factors $SO(p)$ and $SO(q)$, then, noting that $SO(p)$ and $SO(q)$ are simple, it follows by Goursat's lemma \footnote{Goursat's Lemma: Let G,G' be groups, and let H be a subgroup of $G \times G'$ such that the two projections $p_1: H \rightarrow G$ and $p_2 : H \rightarrow G'$ are surjective. Let $N$ be the kernel of $p_2$ and $N'$ the kernel of $p_1$. Clearly, then $N$ may be identified as a normal subgroup of $G$, and $N'$ as a normal subgroup of $G'$. Then the image of $H$ in $G/N \times G'/N'$ is the graph of an isomorphism $G/N \cong G'/N'$.} that $K$ must be $SO(p) \times SO(q)$ (recall the assumption $p \neq q$ and $p,q$ sufficiently large).
\end{proof}

\begin{lem}\label{det-trick}
Let $\rho_0$ be a point  $X_{\Gamma,G}$ with Zariski-dense image in $SO(p) \times SO(q)$. Let $\rho_1, \rho_2, \ldots$ be a sequence of points in $X_{\Gamma,G}$ converging to $\rho_0$. Suppose that for each $n$, $\rho_n$ is conjugate to either $SO(p) \times SO(q)$ or $S(O(p) \times O(q))$. Then, there exists an integer $N$ such that for each $n \geq N$,  $\rho_n$ is conjugate to $SO(p) \times SO(q)$.
\end{lem}

\begin{proof}

For a subgroup $M$ of $G$, let $M^{x} := x M x^{-1}$.  Let $K := S(O(p) \times O(q))$ and let $H := K^\circ = SO(p) \times SO(q)$. 

\begin{claim}\label{claim 1}
Let $\Gamma_2$ be a subgroup of index $2$ in $\Gamma$. Let $\gamma \in \Gamma \setminus \Gamma_2$ represent the (unique) non-identity element of the quotient group $\Gamma/\Gamma_2$. Given $\epsilon > 0$, there exists $\gamma_2 \in \Gamma_2$ such that the product of any subset of the eigenvalues of $\rho_0(\gamma \gamma_2)$ are within an $\epsilon$-distance of $1$.
\end{claim}
\begin{proof}
Since $H$ is Zariski-connected and $\rho_0(\Gamma)$ is Zariski-dense in $H$, we have
$$
\overline{\rho_0(\Gamma_2)} = \overline{\rho_0(\Gamma)} = H .
$$
Let $g := \rho_0(\gamma^{-1})$. Then for any $\R$-open neighborhood $U$ of the identity (matrix) $I_{p+q}$ of $H$, we have $g U \cap \rho_0(\Gamma_2) \neq \emptyset$ (a Zariski-dense subgroup  of $G(\R)$, where $G(\R)$ is compact, is also dense in the real topology.) Then, there exists $\gamma_2 \in \Gamma_2$ such that $\rho_0(\gamma_2) \in gU$ ($ x \in U$), and we have $ x: = g^{-1} \rho_0(\gamma_2)  \in U.$
Now, we may choose $U$ such that the eigenvalues of $x$ are within a distance of $\epsilon$ to the eigenvalues of the identity matrix $I_{p+q}$ which are all equal to $1$ (the eigenvalues of a matrix are continuous functions in the entries of the matrix). Therefore the product of any subset of them are also sufficiently close to $1$, and the claim follows.
\end{proof}

\begin{claim}\label{claim 2}
Let $\rho_1, \rho_2, \ldots, $ be a convergent sequence in $X_{\Gamma,G}$. Suppose that for each $i$, $\overline{\rho(\Gamma_i)} = K^{x_i}$ for some $x_i \in G$. Then there exists a convergent subsequence $(\rho_{n_i})_i$ of the sequence $(\rho_i)_i$, a subgroup $\Gamma_2$ of index $2$ in $\Gamma$, and an element $\gamma \in \Gamma \setminus \Gamma_2$ such that, for each $\gamma_2 \in \Gamma_2$ there exists a subset of eigenvalues of $\rho_{n_i}(\gamma \gamma_2)$ whose product is $-1$.
\end{claim}

\begin{proof}
For each $i$, consider $\rho_i^{-1}(H^{x_i})$ which is a subgroup of index $2$ in $\Gamma$. As there are only finitely many subgroups of index $2$ in $\Gamma$, there are infinitely many $i$ such that 
$$
\rho_i^{-1}(H^{x_i}) = \Gamma_2
$$
for some (fixed) subgroup $\Gamma_2$ of index $2$ in $\Gamma$. Passing onto a convergent subsequence and after re-indexing if necessary, we can assume that the same holds for each $i$.  Let $\gamma \in \Gamma \setminus \Gamma_2$ represent the (unique) non-identity element in the quotient group $\Gamma / \Gamma_2$. Then for each $\gamma_2 \in \Gamma_2$, we have $\rho_i(\gamma \gamma_2) \in K^{x_i} \setminus H^{x_i}$. It follows that 
$x_i^{-1} \rho_i(\gamma \gamma_2) x_i$ (which is in $ K \setminus H$) is of the form 
$$
x_i^{-1} \rho_i(\gamma \gamma_2) x_i  = 
\begin{bmatrix}
A&0\\
0&B
\end{bmatrix}, \; \textrm{where} \; \det(A) = \det(B) = -1.
$$
Clearly then, for each $i$, there exists some subset of eigenvalues of $\rho_i(\gamma \gamma_2)$ whose product is $-1$. The claim follows.
\end{proof}

Now, the proof of lemma follows from Claim \ref{claim 1} and Claim \ref{claim 2}. Suppose on the contrary that for each $i \geq 1$ there exists some $n_i \geq i$ such that $\rho_{n_i}$ is conjugate to $K$. Then, by Claim \ref{claim 2} there exists some element $\gamma \in \Gamma \setminus \Gamma_2$ such that for each $\gamma_2 \in \Gamma_2$, $\rho_{n_i}(\gamma \gamma_2)$ has some subset of eigenvalues whose product is $-1$ (we may ignore the process of passing on to a convergent subsequence). Now $\rho_0(\gamma \gamma_2)$ also has some subset of eigenvalues whose product is $-1$, contradicting Claim \ref{claim 1}.
\end{proof}

\noindent Now, we prove the main result of this section. 
\begin{thm}\label{Deformation-point-II}
Let $\rho_0: \Gamma \to G$ be a representation with dense image in $SO(p) \times SO(q)$. Then $\rho_0$ is a nonsingular point. Moreover, there is an $\R$-open neighborhood around $\rho_0$ contained in the (unique) irreducible component to which $\rho_0$ belongs, such that each point $\rho$ in this neighborhood satisfies precisely one of the following conditions:

\begin{enumerate}[(a)]
 \item
 $\rho$ is dense in $G$.
\item 
$\rho$ is conjugate to $SO(p) \times SO(q)$.

\end{enumerate}
\end{thm}

\begin{proof}
By Proposition \ref{R-open-nbhd}, the representation $\rho_0$ is a nonsingular point and there exists an $\R$-open neighborhood $U$ of $\rho_0$ diffeomorphic to $\R^{d}$, where $d$ is the dimension of the Zariski-tangent space at $\rho_0$. Let $H$ be a maximal proper closed parabolic subgroup of $G$, and let $D$ be an irreducible component of the representation variety $X_{\Gamma,H}$. By \cite[Proposition 3.1]{LL}, the condition on $\rho \in X_{\Gamma,G}$ that $\rho$ is \emph{not} contained in any $G$-conjugate of $D$ is open in the real topology.  As $\rho_0$ is not contained in any parabolic subgroup of $G$, it follows from Proposition \ref{maximals} that we may assume by shrinking $U$ if necessary that the neighborhood $U$ of $\rho_0$ contains representations with image either dense in $G$ or contained in some maximal proper reductive subgroup of $G$. 

Let $m = J(p+q)$ be the constant given by Theorem \ref{larsen}. Then, by Proposition \ref{Deformation-point} each representation in $U$ may be assumed (after shrinking $U$ if necessary) to satisfy either property $(P_m)$ or property $(Q_m)$. 

First, consider those representations $\rho$ in $U$ that satisfy property $(P_m)$. Then $H := \overline{\rho(\Gamma)}$ acts absolutely irreducibly on $V$, hence $H$ must be reductive. We claim that its derived group $[H^\circ,H^\circ]$ also acts absolutely irreducibly on $V$.  As $H$ acts absolutely irreducibly on $V$, we may extend scalars to $\C$ and consider $H$ and $V$ over $\C$. By Theorem \ref{larsen}, there exists a subgroup $K$ of $H$ such that 
$$
 K =  [H^\circ, H^\circ]Z(K)
$$
and the index $[H:K] \leq J(n)$, where $J(n)$ is a constant that depends only on $n$. Let $\Lambda$ be the intersection of all subgroups of index at most $m$ in $\Gamma$. Since $\rho$ satisfies the property $(P_m)$, the closed subgroup $\overline{\rho(\Lambda)}$ also acts absolutely irreducibly on $V$. Moreover, since $\Lambda$ is of finite index in $\Gamma$ and $\rho(\Lambda) \subseteq K$,  $\overline{\rho(\Lambda)}$ contains $H^\circ$: $H^\circ \subseteq \overline{\rho(\Lambda)} \subseteq K$. Therefore,
$$
 [H^\circ,H^\circ] \subseteq \overline{[\rho(\Lambda),\rho(\Lambda)]} \subseteq [K,K] = [H^\circ,H^\circ].
$$

Let $L := \overline{[\rho(\Lambda),\rho(\Lambda)]}$. Then $L = [H^\circ,H^\circ]Z(L)$. Since $Z(L)$ acts by scalars on $V$ (Schur's lemma), any proper subrepresentation of $[H^\circ,H^\circ]$ is also a proper sub-representation of $L$. It follows that $[H^\circ,H^\circ]$ acts irreducibly on $V$. The claim follows.

Now, let $\rho_1, \rho_2, \ldots $ be a sequence of representations in $U$ converging to $\rho_0$ such that for each $n \geq 1$, $\rho_n$ satisfies the property $(P_m)$. We claim that there exists some positive integer $N$ such that for each $n \geq N$, $\rank \rho_n \geq \rank \rho_0$. Suppose on the contrary that for each $N \geq 1$, there exists some $n_N \geq N$ such that $\rank_{\ssim} \rho_{n_N}  < \rank_{\ssim} \rho_0$. By passing onto a convergent subsequence and after changing the notation, we may assume that for each $n \geq 1$, $\rank_{\ssim} \rho_n < \rank_{\ssim} \rho_0$. Note that $\overline{\rho_i(\Gamma)}$ is reductive for each $i$, since a linear algebraic group with a faithful irreducible representation is reductive. Let $\mathcal{D}\rho_i:= \overline{[\rho_i(\Gamma), \rho_i(\Gamma)]} = [\overline{\rho_i(\Gamma)}, \overline{\rho_i(\Gamma)}]$ which is semisimple. As $\rho_i$ acts absolutely irreducibly on $V$, it follows  that $\mathcal{D}\rho_i$ also acts absolutely irreducibly on $V$. By Proposition \ref{finite-semisimple}, there are finitely many possibilities $\{ S_1, \ldots, S_r \}$ for  $\mathcal{D} \rho_i$ up to conjugation. Consider the characteristic morphism $\chi : G \to \A^{p+q-1}_\C$ (Refer to the discussion preceding Lemma \ref{Image-rank}.) It follows from Proposition \ref{Maximal-avoidance} that there exists an element $\gamma \in \Gamma$ such that $\chi(\rho_0(\gamma)) \in \overline{\chi(\rho_0(SO(p) \times SO(q))} \setminus \bigcup_{i=1}^{r}\overline{\chi(S_i)}$, so that  $\rho_n(\gamma)$ cannot converge to $\rho_0(\gamma)$. Therefore, there exists $N \geq 1$ such that for all $n \geq N$, $\rank_{\ssim} \rho \geq \rank_{\ssim} \rho_0$. In other words, there is an open neighborhood of $\rho_0$ consisting of representations $\rho$ such that $\rank_{\ssim} \rho \geq \rank_{\ssim} \rho_0$. Shrinking $U$ further if necessary, we may suppose that this property on rank holds for representations in $U$.  By Proposition \ref{Rank-irreducible}, representations in $U$ that satisfy the property  $(P_m)$ must have rank strictly less than the $\rank SO(p) \times SO(q)$ unless they are dense in $G$. It follows that representations that such representations are indeed dense in $G$. 

Now consider the representations $\rho$ in $U$ that satisfy the property $(Q_m)$.  By Lemma \ref{absolute-reducible}, $\overline{\rho(\Gamma)}$ cannot be a contained in a \emph{proper} closed subgroup of $S(O(p) \times O(q))$ or its conjugate, other than $SO(p) \times SO(q)$ or its conjugate. By Lemma \ref{det-trick} we may further shrink $U$ if necessary, so that representations in $U$ that satisfy the property $(Q_m)$ are conjugate to $SO(p) \times SO(q)$.

\end{proof}

\section{Main Theorem}\label{main-thm}

\begin{lem}\label{tangent-space}
Let $T$ be a semisimple operator of finite order  acting on an $n$-dimensional complex vector space $V$. Let $S$ be the set of eigenvalues of $T$ and $t_{\lambda}$ be the multiplicity of the eigenvalue $\lambda \in S$. The dimension of eigenspace corresponding to the eigenvalue $1$ under the induced action of $T$ on $\wedge^{2} V$ is given by,
\begin{equation}\label{equ 4.1}
\dim (\wedge^{k} V)^T = {t_1 \choose 2 } + {t_{-1} \choose 2 } +
\frac{1}{2}\sum_{\lambda \in S - \{ 1, -1\} } t_{\lambda} \cdot  t_{\bar{\lambda}}.
\end{equation}
Note, $t_1 + t_{-1} + \sum_{\lambda \in S - \{ 1, -1\} } t_{\lambda} \cdot  t_{\overline{\lambda}} = n$.
\end{lem}

\begin{proof}
Let $V_\lambda$ be the eigenspace corresponding to the eigenvalue $\lambda$.  For each eigenvalue $\lambda$, fix a basis $ v^{\lambda}_1, \ldots, v^{\lambda}_{t_{\lambda}} $ of
$V_{\lambda}$.

The lemma follows from the observation that vectors of the form $ v^{1}_{i} \wedge v^{1}_{j}$ where $1 \leq i < j \leq t_1$,  $v^{-1}_{i} \wedge v^{-1}_{j}$ where $1 \leq i < j \leq t_{-1}$, and $ v^{\lambda}_{i} \wedge v^{\overline{\lambda}}_{j}$ where $1 \leq i \leq t_\lambda$ and $1 \leq j \leq t_{\overline{\lambda}}$ and $i \leq j$, constitute a basis of $+1$-eigenspace of $(\wedge^2 V)^T$.
\end{proof}

\begin{cor}\label{error-term}

Let $x_1, x_2, \ldots, x_m$ be a set of elements of (finite) order $d_1, \ldots ,d_m$ in $SO(p,q)$. Suppose that the multiplicity of each eigenvalue of $x_i$ under the adjoint representation of $SO(p,q)$ is $\frac{p+q}{d_j} + O(1)$, where the implicit constant depends only on $d_i$. Then,
\[
\sum_{j =1 }^{m} \left( \frac{\dim \so(p,q)}{d_j} - \dim \so(p,q)^{\langle x_j \rangle} \right) = - \; O(p+q).
\]
where the implicit constant depends only on $d_1,\ldots,d_m$.
\end{cor}

\begin{proof}
Let $\g := \so(p,q)$. A brief computation using Lemma \ref{tangent-space} establishes the corollary, noting that 
$\dim (\bigwedge^{2} \g)^{x_j} = \frac{(p+q)^2}{2d_j} - \frac{p+q}{d_j} + O(1)$:
\begin{align*}
\sum_{j =1 }^{m} \left( \frac{\dim \g}{d_j} - \dim \g^{\langle x_j \rangle} \right)
&=  \sum_{j=1}^{m} \left( \frac{(p+q)(p+q-1)}{2d_j} -  \frac{(p+q)^2}{2d_j} \right)  \\
&   -(p+q) \sum_{j=1}^{m}  O(1) \\
&= - \;  O(p+q). 
\end{align*}
\end{proof}

\begin{prop}\label{dimension-compare}
Let $\Gamma$ be a Fuchsian group (cocompact and oriented) with the presentation given by (\ref{pre}). Let $G := SO(p,q)$ and let $\g := \so(p,q)$ its Lie algebra. Let $\rho_0 : \Gamma \to G$ with Zariski dense image in $H := SO(p) \times SO(q)$. Suppose that the following inequality holds:
\begin{equation}\label{check}
\begin{split}
 - \chi(\Gamma)(\dim G - \dim H )  &>  2g   +   m   +  (3m/2 +1)  \; \rank H \\ 
 & - \sum_{i=1}^{m} \left( \frac{ \; \dim   G}{d_i}  -   \; \dim   \g^{\la x_i \ra } \right).
\end{split}
\end{equation}
Then $\rho_0$ is a nonsingular point; hence it belongs to a unique irreducible component $C$ of $X_{\Gamma,G}$. Furthermore, there exists an $\R$-open neighborhood $U$ of $\rho_0$ in $C$ such that the Zariski-closure of
$$
 \{ \rho \in U :  \overline{\rho(\Gamma)} \; \textrm{is conjugate to} \; H \}
$$
constitutes a proper closed subvariety in $C$.
\end{prop}

\begin{proof}
Let $t_{G}$, resp. $t_{H}$, denote the dimension of the Zariski tangent space at $\rho_0$ considered as a point in $X_{\Gamma,G}$, resp., $X_{\Gamma,H}$. 

By Theorem \ref{Deformation-point-II}, $\rho_0$ is  a nonsingular point in $X_{\Gamma,G}$, and therefore it belongs to a unique irreducible component $C$ of $X_{\Gamma,G}$; furthermore, by the same theorem, there exists an $\R$-open neighborhood $U$ of $\rho_0$ such that for each $\rho \in U$, either $\rho(\Gamma)$ is dense in $G$ or dense in a conjugate of $H$. In particular, each representation with Zariski-dense image in a conjugate of $H$ is also nonsingular in $X_{\Gamma,G}$, hence representations in $U$ conjugate to $\rho_0$ belong to $C$ \textit{only}. But, on the other hand, $\rho_0$ considered as a point in $X_{\Gamma,H}$ may not be nonsingular.

Consider all the irreducible components $D_1, \ldots ,D_r$ of $X_{\Gamma,H}$ such that each $D_i$ contains some conjugate of $\rho_0$ that also belongs to $U$. Let $D := \bigcup_{i=1}^r D_i$. Now, we reproduce a part of the proof of Theorem 3.4 of Larsen and Lubotzky \cite{LL} which essentially gives a criterion for the Zariski-closure of $D \cap U$ to be a proper closed subvariety in $C$: For each irreducible component $D_i$ consider the conjugation morphism $\chi_{i}: G \times D_i \to X_{\Gamma,G}$. The fibers of this morphism have dimension at least dim $H$. Indeed, the action of $H^\circ$ on $G \times D_i$ given by 
\[
 h \cdot (g,\rho) = (gh^{-1}, h \rho h^{-1})
\]
is free, and $\chi_{i}$ is constant on the orbits of the action. Therefore, the image of $\chi_{i}$ has dimension at most dim $D_i$ + dim $G$ - dim $H$. It follows that, if 
\begin{equation}\label{dim-ineq}
t_G  \; - \;  \dim G \; > \; \dim D \;  -  \; \dim H 
\end{equation}
then the dimension of the image of $\chi_i$ is less than the dimension of $C$, hence the Zariski-closure of the set of representations in $U$ belonging to some conjugate of $D$ constitutes a proper closed subvariety of $U$. Therefore, to complete the proof, it suffices to show that the inequality holds. 

For every finitely generated group $\Gamma$, every $\R$-algebraic group $H$ with Lie algebra $\h$, and every representation $\rho : \Gamma \to H$ with dense image, an upper bound on the dimension of the Zariski tangent space at $\rho \in X_{\Gamma,H}$ is given by \cite[Proposition 2.1]{LL}:
\begin{equation}\label{upb1}
t_H \leq (1 - \chi(\Gamma))\; \textrm{dim} \;  H + (2g + m + \; \rank  H) + \frac{3}{2}  m   \rank H.
\end{equation}
Since the dimension of each $D_i$ is bounded above by $t_H$, it follows that the upper bound for $t_H$ given by (\ref{upb1}) is also an upper bound for $\dim D$.  

From the proof of Proposition \ref{R-open-nbhd}, we have  $(\g^{*})^{\Gamma} = 0$, so from the formula (\ref{tas}), the dimension of $\rho_0$ in $X_{\Gamma, G}$ is given by:
\begin{equation}\label{tas1}
t_G = \dim Z^1(\Gamma,Ad \circ \rho_0) = (1 - \chi(\Gamma))   \; \dim   G +  \sum_{i=1}^{m} \left( \frac{ \; \dim   G}{d_i}  -   \; \dim   \g^{\la x_i \ra } \right).
\end{equation}
Now, using the expressions (\ref{upb1}) and (\ref{tas1}) in (\ref{dim-ineq})  we obtain the inequality of the hypothesis, and the proposition follows.
\end{proof}

\noindent Finally, we now prove the main result of the paper.
\begin{thm}\label{main-result}
For every Fuchsian group $\Gamma$ and for every special orthogonal group $SO(p,q)$, where $p \neq q$, $q \neq p+1$ and $p,q$ sufficiently large,
\begin{equation*}
 \dim   X^{\epi}_{\Gamma,SO(p,q)}  =  (1- \chi(\Gamma))   \; \dim  SO(p,q) + O(\rank SO(p,q)),
\end{equation*}
where the implicit constants depend only on $\Gamma$.
\end{thm}

\begin{proof}
Let $G := SO(p,q)$ and let $\g := \so(p,q)$ its Lie algebra. As noted in \S \ref{intro}, a result of Larsen and Lubotzky \cite[Proposition 2.1]{LL} essentially gives an upper bound on the dimension of $X^{\epi}_{\Gamma,G}$. Thus, it then suffices to establish a lower bound. Note, $G$ is split when $p = q$ or $q = p+1$, and the estimates in these cases are treated in \cite{LL} as explained briefly in \S \ref{intro}. 

Proposition 3.1 \cite{LL} ensures that there exists representations $\rho : \Gamma \to SO(p)$ and $\tau : \Gamma \to SO(q)$, such that $\rho(\Gamma)$ is Zariski dense in $SO(p)$ and $\tau(\Gamma)$ is Zariski-dense in $SO(q)$. Moreover, these representations may be chosen so that the multiplicity of every $d_i^{\textrm{th}}$ root of unity as an eigenvalue of $\sigma(x_i)$, resp., $\tau(x_i)$ is of the form $p/d_i + O (1)$, resp., $q/d_i + O(1)$ where the implicit constant depend only on $d_i$. Consider the product representation $\rho_0 = (\sigma,\tau) : \Gamma \rightarrow SO(p) \times SO(q)$.  By Goursat's lemma, $\rho_0$ is dense in $K$. Clearly, the multiplicity of every $d_i^{\textrm{th}}$ root of unity as an eigenvalue of $\rho_0(x_i)$ is of the form $(p+q)/d_i + O (1)$, where the implicit constant depend only on $d_i$. 

By Lemma \ref{R-open-nbhd}, $\rho_0$ is a nonsingular point, and by Theorem \ref{Deformation-point-II} there exists an $\R$-open neighborhood $U$ of $\rho_0$ consisting of representations $\rho$ such that  $\overline{\rho(\Gamma)} = G$ or $\overline{\rho(\Gamma)}$ is conjugate to $SO(p) \times SO(q)$. By Proposition \ref{dimension-compare}, representations of the latter type constitute a proper closed subvariety in $C$, for the  inequality (\ref{check}), using Corollary \ref{error-term}, takes the following simple form which evidently holds when $p$ and $q$ are sufficiently large. 
\[
\alpha + \beta (p+q) < -\chi(\Gamma)pq,
\]
where $\alpha$ and $\beta$ are constants that depend only $d_i$.  Hence representations of the former type constitute an nonempty $\R$-open subset in $C$. Since a nonempty $\R$-open subset in $C$ is Zariski dense, it follows that $C$ is contained in $X^{\epi}_{\Gamma,G}$. By Corollary \ref {error-term} and equation (\ref{tas}), we have
$$
\dim X^{\epi}_{\Gamma,G} \geq \dim C = \dim Z^1(\Gamma, \Ad \circ \rho_0) = (1 - \chi(\Gamma))   \; \dim   G  \; + \; O(\rank G).
$$
\end{proof}

\section{Final Remarks}\label{final-rmk}

Notice that the results of section \S \ref{deformation-point} should extend easily, with appropriate modifications, to the case of noncompact nonsplit symplectic groups $Sp(p,q)$. For instance, skew-symmetric bilinear forms instead of symmetric bilinear forms should be considered and the representation $\rho_0 : \Gamma \to Sp(p,q)$ with Zariski dense image in $Sp(p) \times Sp(q)$ may be taken as the deformation point. On the other hand, these results do not generalize easily to the unitary groups $SU(p,q)$. For instance, if we consider the deformation point as the representation $\rho_0 : \Gamma \to SU(p,q)$ with dense image in $SU(p) \times SU(q)$, then certainly Proposition \ref{R-open-nbhd} does not hold, for the dimension of the centralizer of $SU(p) \times SU(q)$ in $SU(p,q)$ is $1$, therefore Weil's result on nonsingularity  of point is not applicable to conclude that $\rho_0$ is a nonsingular point in $X_{\Gamma,SU(p,q)}$ (see Theorem \ref{real-points}); as it is evident by now from the proof of Theorem \ref {Deformation-point-II}, nonsingularity of the chosen deformation point $\rho_0$ is crucial.

\end{document}